\documentclass[a4paper,twoside]{amsart}
\usepackage{a4}
\usepackage{psfrag}
\usepackage{graphicx}
\usepackage{lineno}
\newtheorem{thm}{Theorem}[section]
\newtheorem{cor}[thm]{Corollary}
\newtheorem{lem}[thm]{Lemma}
\newtheorem{prop}[thm]{Proposition}
\theoremstyle{definition}
\newtheorem{defn}[thm]{Definition}
\theoremstyle{remark}
\newtheorem{rem}[thm]{Remark}

\numberwithin{equation}{section}

\newcommand{\Real}{\mathbb R}

\begin{document}

\title[Hopf bifurcation at infinity]{Hopf bifurcation at infinity and dissipative vector fields of the plane}%
\author[B. Alarc\'on]{Bego\~{n}a Alarc\'on}%
\address[B. Alarc\'on]{Departamento de Matem\'atica Aplicada, Universidade Federal Fluminense, Rua M\'ario Santos Braga S/N, CEP 24020-140 Niter\'oi-RJ, Brasil }%
\email{balarcon@id.uff.br}%
\curraddr{}

\author[R. Rabanal]{Roland Rabanal}%
\address[R. Rabanal]{Departamento de Ciencias, Pontificia Universidad Cat\'{o}lica del Per\'{u},  Av. Universitaria 1801, Lima~32;  Per\'{u}}%
\email{rrabanal@pucp.edu.pe}%
\curraddr{}


\date{\today}%
\dedicatory{Dedicated to the memory of Carlos Guti\'errez,\\
on occasion of the 05th anniversary of his death.}%
\commby{}
\begin{abstract}
This work deals with one--parameter families of differentiable (not necessarily $C^1$) planar vector fields for which the infinity reverses its stability as the parameter goes through zero. These vector fields are defined on the complement of some compact ball centered at the origin and have isolated singularities. They may be considered as linear perturbations at infinity of a vector field with some spectral property, for instance, dissipativity. We also address the case concerning linear perturbations of planar systems with a global period annulus. It is worth noting that the adopted approach is not restricted to  consider vector fields which are strongly dominated by the linear part. Moreover,  the  Poincar\'e compactification is not applied in this paper.
\end{abstract}

\maketitle


\section{Introduction}%
The main purpose of this paper is to deepen the study of qualitative theory of differential equations induced by differentiable vector fields of the plane which are not necessarily of class $C^1$. For instance, problems addressed in \cite{MR2096702, MR2287882,MR2266382} include  global injectivity, uniqueness of solutions and asymptotic stability of the point at infinity of the Riemann sphere $\Real\cup \{\infty\}$. These works, like many others in the literature, extend to the differentiable case previous results concerning $C^1-$vector fields.  Thus, the present paper pretends to extend to the differentiable case the work in  \cite{MR2257430}.  The authors of \cite{MR2257430} guarantee the change of stability at infinity of a one--parameter family of $C^1-$vector fields as the parameter goes through zero.
This behaviour was presented as Hopf bifurcation at infinity.
\par
Before \cite{MR2257430} the study of Hopf bifurcation at infinity had been approached only from two points of view. One deals with families which have finitely many parameters and are made up of continuous perturbations of linear systems. The other point of view deals with polynomial families. The former case uses the strong domination imposed by the linear part of the vector field on the continuous nonlinear term, which can have sub--linear growth  \cite{MR1028236,MR1124985,MR1364308,MR1446100,MR1780452}. The later point of view involves standard Poincar\'e compactification \cite{MR967475,MR1213941,MR1231467,MR1421060}.
\par
Both points of view are related to the existence of limit cycles of arbitrarily large amplitudes. Particularly, the case of polynomial systems can be associated to the sixteenth Hilbert problem, see for instance \cite{MR934515,MR909943,MR2093918,MR3130553}. In contrast, a different point of view was given in \cite{MR2257430}, where the authors are only interested in the change of  stability at infinity. This point of view has been strongly articulated in other works related to planar systems induced by differentiable vector fields \cite{MR1339178,MR2040002,MR2287882,MR2257430,MR3062761,MR3223368}. The main motivation of this approach was the definition of Hopf bifurcation at infinity given in \cite{MR909943} (see also \cite{MR934515,MR2142367}). Concretely, it is said in \cite{MR909943} that the polynomial family of vector fields $\{Z_{\mu}:-\epsilon<\mu<\epsilon\}$ presents a Hopf bifurcation at infinity when $\mu$ cross $0$ if\, \lq the infinity changes its stability\rq.
\par
The paper \cite{MR2257430} deals with a class of one--parameter family ${\{Z_{\mu}\; ;\; -\epsilon<\mu<\epsilon\}}$ of planar $C^1$ vector fields  defined on the  complement of an open ball centered at origin which are also free of singularities. The authors find the change of stability at infinity when the sign of  $\displaystyle{\int_{\Real^2}\; div(\hat{Z_\mu})}$ varies as the parameter crosses zero. In the context of \cite{MR2257430}, $\hat{Z_\mu}: \Real^2 \to \Real^2$ is any  $C^1-$ extension of $Z_\mu$  with divergence  Lebesgue almost--integrable (see Section \ref{sec_prelim} for details). The families considered in   \cite{MR2257430} are neither polynomial nor a perturbation of linear systems, consequently the work in \cite{MR2257430} contains the two point of view described above.
\par
We consider a one--parameter family $\{X_{\mu}(z)=X(z)+\mu z, \;\mu\in \Real\}$, where $X$ has isolated singularities and belongs to a class of differentiable  (not necessarily of class $C^1$) planar vector fields defined on the complement of some compact ball centered at the origin. The main motivation to address the existence of Hopf bifurcation at infinity for this family is the work given in \cite{MR2096702} and \cite{MR2287882}. In those papers the authors   prove the uniqueness of the positive trajectory starting at every point $p\in \Real^2$ due some conditions on the eigenvalues of the  Jacobian matrix of the differentiable vector field. In addition, results in \cite{MR2287882} allow us to focus on the integral $\displaystyle{\int_{\Real^2}\; div(\hat{X_\mu})}$ in order to find a change of stability at infinity. In this context, since the qualitative change at infinity of the family  $\{X_{\mu}(z)=X(z)+\mu z, \;\mu\in \Real\}$  may be exhibit from global extensions  $\hat{X_\mu}: \Real^2 \to \Real^2$, the main obstruction to describe the bifurcation is the existence of unbounded sequences of singular points. In order to avoid this phenomenon, we impose some extra condition on the extension $\hat{X_\mu}$ which imply that  every vector field of the family induces an injective map.  Consequently, every vector field of the family has at most one singular point in the plane.
\par
Observe that the one--parameter family considered in this paper is not necessarily a perturbation of linear systems in the sense of \cite{MR1028236,MR1124985,MR1364308,MR1446100,MR1780452}.
Moreover, \mbox{Corollary \ref{cor:1}} complements the results presented in~\cite{MR2257430} for $C^1-$vector fields because in our work the existence of isolated singularities of $X$  is allowed.
\par


\section{Notation and definitions} \label{sec_prelim}

Let $\overline{D}_\sigma=\{z\in\mathbb{R}^2:||z||\leq\sigma\}$ be the compact ball bounded by
$\partial\overline{D}_\sigma=\{z\in\mathbb{R}^2:||z||=\sigma\}$ with $\sigma>0$, and suppose that $X:\mathbb{R}^2\setminus\overline{D}_{\sigma}\to\mathbb{R}^2$ is a differentiable vector field, defined on the complement  of $\overline{D}_{\sigma}$ on $\mathbb{R}^2$.
Here, the  term  differentiable means \textit{Frechet differentiable} at each point $z\in\mathbb{R}^2\setminus\overline{D}_{\sigma}$ and $X^{\prime}(z):\mathbb{R}^2\to\mathbb{R}^2$ is the respective derivative. This derivative is defined as the bounded linear operator induced by the standard Jacobian matrix at $z$. This matrix is denoted by $DX(z)$, so $\mbox{det}(DX(z))$ and $\mbox{Trace}(DX(z))$ are the typical Jacobian determinant and  divergence, respectively.
On the other hand, if the point $q\in\mathbb{R}^2\setminus\overline{D}_{\sigma}$ is kept fixed, then a \textit{trajectory} of $X$ starting at $q$ is defined as the integral curve $I_q\ni t\mapsto\gamma_q(t)\in \mathbb{R}^2\setminus\overline{D}_{\sigma}$, determined by a maximal solution of the Initial Value Problem $\dot{z}=X(z),~ z(0)=q$.
Of course, it means that $\frac{d }{dt}\gamma_q(t)=X\big(\gamma_q(t)\big),\forall t\in I_q$ and $\gamma_q(0)=q$. In this context, it is useful to have a term that refers to the image of the solution.
Hence, we define the \textit{orbit} of  $\gamma_q$ to be the set $\{\gamma_q(t):t\in I_q\}$, and
we identify the trajectory with its orbit.
Similarly, ${\gamma^+_q}=\{\gamma_q(t):t\in I_q\cap[0,+\infty)\}$ and ${\gamma^-_q}=\{\gamma_q(t):t\in (-\infty,0]\cap I_q\}$ are the \textit{semi--trajectories} of $X$, and they are called \textit{positive} and \textit{negative}, respectively.
In consequence, $\gamma^-_q\cup \gamma^+_q=\gamma_q$, and each trajectory has its two limit sets, $\alpha(\gamma_q^-)$ and $\omega(\gamma_q^+)$.
These limit sets are well defined in the sense that they only depend on the respective solution.
And lastly, such a vector field induces a well defined \textit{positive semi--flow} (respectively \textit{negative semi--flow}), if the condition $q\in\mathbb{R}^2\setminus\overline{D}_{\sigma}$ implies the existence and uniqueness of $\gamma_q^{+}$ (respectively $\gamma_q^{-}$).
\par
The trajectories of a vector field may be unbounded.
One way  to obtain some information about the behavior of such solutions is to compactify the plane, so that the vector field is extended to a new manifold that contains the \lq points at infinity\rq.\,
In this context, the so called Alexandroff compactification has been most successful in the study of planar systems induced by $C^1$-vector fields, not necessarily polynomial (see for example \cite{MR0176180,MR934515,MR2287882,MR2257430,MR3062761}).
To describe the results, $\mathbb{R}^2$ is embedded in the Riemann sphere $\mathbb{R}^2\cup\{\infty\}$.
Consequently, $(\mathbb{R}^2\setminus\overline{D}_\sigma)\cup\{\infty\}$  is the subspace of  $\mathbb{R}^2\cup\{\infty\}$ with the induced topology, and \lq infinity\rq   ~refers to the point $\infty$ of $\mathbb{R}^2\cup\{\infty\}$.
Moreover, a vector field $X:\mathbb{R}^2\setminus\overline{D}_\sigma\to\mathbb{R}^2\setminus\{0\}$ (without singularities) can be extended to a map
\[
\hat{X}:\big(({\mathbb{R}^2\setminus\overline{D}_\sigma})\cup\{\infty\},\infty\big)\longrightarrow(\mathbb{R}^2,0)
\]
(which takes $\infty$ to $0$).
In this manner, all questions concerning the local theory of isolated singularities of planar vector fields can be formulated and examined in the case of the extended vector field $\hat{X}$, which coincides with $X$ on $\mathbb{R}^2\setminus\overline{D}_\sigma$.
For instance, if $\gamma_p^+\subset\mathbb{R}^2\setminus\overline{D}_\sigma$ is an unbounded semi--trajectory of $X:\mathbb{R}^2\setminus\overline{D}_\sigma\to\mathbb{R}^2$  with empty $\omega-$limit set, then we declare that ${\gamma_p^+}$ \textit{goes to infinity}, and we write $\omega({\gamma_p^+})=\infty$.
Similarly, $\alpha({\gamma_p^-})=\infty$ denotes that $\gamma_p^-$ \textit{comes from infinity}, and it means that $\gamma_p^-\subset\mathbb{R}^2\setminus\overline{D}_\sigma$ is an unbounded semi--trajectory whose $\alpha-$limit set is empty.
Therefore, it is also possible to talk about the phase portrait of $X$ in a neighborhood of $\infty$, as shown \cite[\mbox{Proposition 29}]{MR2287882}.

Throughout this paper, given $C\subset\mathbb{R}^2$, a closed (compact, no boundary) curve ($1-$manifold), $\overline{D}(C)$ (respectively $D(C)$) denotes the compact disc (respectively open disc) bounded by $C$.
Thus, the boundaries $\partial \overline{D}(C)$ and  $\partial D(C)$ are equal to $C$ besides homeomorphic to the circle $\partial D_1=\{z\in\mathbb{R}^2:||z||=1\}$.


\subsection{Index at infinity}

The  index at infinity was defined firstly for $C^1-$vector fields in \cite{MR1339178} and lately in \cite{MR2287882} for differentiable vector fields (not necessarily of class $C^1$). The relationship between this index and the stability of the point at infinity of the vector field is also stated in those works in both cases.

\begin{defn}[\cite{MR2287882}] Consider $X:\mathbb{R}^2\setminus\overline{D}_\sigma\to\mathbb{R}^2$ be a differentiable vector field. The \textit{index of $X$ at infinity}, denoted by  $\mathcal{I}(X)$,  is the number of the extended line $[-\infty,+\infty]$ given by
\[
\mathcal{I}(X)={\int_{\mathbb{R}^2}}\mbox{\rm
Trace}(D\widehat{X})dx\wedge dy,
\]
where $\widehat{X}:\mathbb{R}^2\to\mathbb{R}^2$ is a global differentiable vector field such that:
\begin{itemize}
    \item In some $\mathbb{R}^2\setminus\overline{D}_s$ with $s\geq\sigma$, both $X$ and $\widehat{X}$ coincide.
    \item $z\mapsto\mbox{\rm Trace}(D\widehat{X}(z))$ is Lebesgue almost--integrable on $\mathbb{R}^2,$ in the sense {of \cite{MR2287882}}.
\end{itemize}
\end{defn}
This index is a well--defined number in $[-\infty,+\infty]$, and it does not depend on the pair  $(\widehat{X},s)$ as shown in \cite[\mbox{Lemma 12}]{MR2287882} (see also \cite{MR2257430} for the $C^1$ case).


\subsection{Hopf bifurcation at infinity} \label{def:hopfbif}%

The definition of Hopf bifurcation at infinity for differentiable vector fields presented in this paper is based on the definition of Hopf bifurcation at infinity for $C^1$ vector fields given in \cite{MR1339178} and the definition of attractor/repellor given in \cite{MR2287882} for differentiable vector fields.
In \cite{MR2287882}, the authors relate the index at infinity of a differentiable vector field with the stability of the point at $\infty$ of the Riemann sphere.

\begin{defn}[\cite{MR2287882,MR3062761}]\label{def:atractor repellor}
The point  at infinity $\infty$ of the Riemann sphere $\mathbb{R}^2\cup\{\infty\}$ is an \textit{attractor} (res\-pec\-ti\-vely, a \textit{repellor}) for 
$X:\mathbb{R}^2\setminus\overline{D}_\sigma\to\mathbb{R}^2$ if:
\begin{itemize}
    \item There is a sequence of closed curves, transversal  to $X$ and tending to infinity.
        It means that for every $r\geq\sigma$ there exists a closed curve $C_r$ such that $D(C_r)$ contains $D_r$ and $C_r$ has transversal contact to each small local integral curve of $X$ at any $p\in C_r$.
    \item For some $C_s$ with $s\geq\sigma,$ all the trajectories $\gamma_p$ starting at a point $p\in
    \mathbb{R}^2\setminus\overline{D}(C_s)$ satisfy $\omega({\gamma^+_p})=\infty$, that is $\gamma_p^+$ goes to infinity (respectively, $\alpha({\gamma^-_p})=\infty$, that is  $\gamma_p^-$ comes from infinity).
\end{itemize}
\end{defn}

Notice that in the $C^1$ case, \mbox{Definition \ref{def:atractor repellor}} is equivalent to saying that the vector field $\hat{X}$ induced  by $X$ on the Riemann sphere is locally topologically equivalent in an open neighborhood  of the infinity to $z\mapsto -z$ (respectively,  $z\mapsto z$) at the origin \cite{MR1339178,MR2040002,MR2257430}.
\par
The authors of \cite{MR2287882} relate the index at infinity of a differentiable vector field with the stability of the point at $\infty$ of the Riemann sphere.

\begin{defn}[\cite{MR2257430}]\label{def:hopf}%
We will say that the family of differentiable vector fields $\big\{X_{\mu}:\mathbb{R}^2\setminus{\overline{D}_{\sigma}}\to\mathbb{R}^2\; :\;-\varepsilon<\mu<\varepsilon\big\}$ has at $\mu=0$ a \textit{Hopf bifurcation at $\infty$} if the following two conditions are satisfied:
\begin{itemize}
  \item For $\mu<0$ (resp. $\mu>0$), the vector field $X_{\mu}$ has a repellor at $\infty$, and for $\mu>0$ (resp. $\mu<0$), the vector field $X_{\mu}$ has an attractor at $\infty$.
  \item The vector field $X_{\mu}$ has no singularities in $\mathbb{R}^2\setminus{\overline{D}_{s}}$, for $-\varepsilon<\mu<\varepsilon$ and some $s\geq\sigma$.
\end{itemize}
\end{defn}

In order to capture the essential features of \mbox{Definition \ref{def:hopf}}, we remark that the element   $X_0$ might be unstable in the sense that the vector field $\hat{X_0}$ induced by $X_0$ on the Riemann sphere might be locally topologically equivalent in an open neighborhood  of the infinity to a center at the origin.
In this context, $X_0$ has first degree of instability \cite{MR0344606} in the sense that it is unstable whereas any vector field in  $\{X_{\mu}\}$, with $\mu$ sufficiently close to $0$, is either stable (i.e $\infty\in\{attractor, repellor\}$) or topologically equivalent to $X_0$  (see  Figure \ref{fig:mainf proposition}). In Section \ref{sec:free eigen},  we also impose some condition on $X_0$ in order to force a specific family $\{X_{\mu}\}$ to present a Hopf bifurcation at $\infty$ as a topological characterization of the bifurcation when $\infty$ reverses its stability.
\par
The authors of \cite{MR2257430} relate  the existence of a \emph{Hopf bifurcation at infinity} for a family of $C^1-$vector fields with the \emph{index at infinity} of each vector field of the family.

\begin{thm}[\cite{MR2257430}] \label{teoAGG07} Let $\, \{ Z_{\mu} :\Real^2\setminus D_{r}\; ;\; -\varepsilon < \mu
< \varepsilon\} \,$ be a family of $C^1-$vector fields such that:
\begin{enumerate}
 \item[(1)] $z\mapsto\mbox{\rm Trace}\big(DZ_{\mu}(z)\big)$ is Lebesgue almost--integrable on $\mathbb{R}^2\setminus{D_{r}}$,
  \item[(2)] $\mathcal{I}(Z_{\mu})$ is well defined, this index belongs to $[-\infty,+\infty]$, and
  \[
  {\int_{\sigma}^{+\infty}}\Upsilon_{\mu}(r) dr=+\infty,
  \]
  where $\Upsilon_{\mu}(r)=\inf\big\{||Z_{\mu}(z)||:||z||=r\big\}$,
  \item[(3)] $\mu\neq0$ implies that the product $\mu \cdot\mathcal{I}(Z_{\mu}) >0$,
  \item[(4)] $Z_{\mu}$ has no singularities, and the \textit{Poincar\'e index} at $\infty$ of the vector field
  \[\hat{Z}_{\mu}:\big(({\mathbb{R}^2\setminus{D_r}})\cup\{\infty\},\infty\big)\longrightarrow(\mathbb{R}^2,0)\]
  (which extends $Z_{\mu}$ at $\infty$) is less than or equal to $1$.\\

\end{enumerate}
Then $ \mu = 0 $ is a Hopf bifurcation at $ \infty $ of the family $ \{ X_{\mu}\} $.
\end{thm}

\mbox{Theorem \ref{teoAGG07}} is the one we extend to the differentiable case in \mbox{Theorem \ref{main 1}}.

\section{Hopf bifurcation at infinity for dissipative vector fields}\label{sec:3}%

A differentiable planar vector field $X=f\frac{\partial}{\partial x}+g\frac{\partial}{\partial y}$ is called \emph{dissipative} on a region $D$ if $div(X)=\frac{\partial f}{\partial x}+\frac{\partial g}{\partial y}< 0$ on $D$ and the equality only holds on a set with Lebesgue measure zero. See \cite{MR1094380} for more details.

Let $\mathcal{H}(2,\sigma)$ be the set of  differentiable vector fields $X:\mathbb{R}^2\setminus{\overline{D}_\sigma}\to\mathbb{R}^2$ which are dissipative and have strictly positive Jacobian determinant. This means that all vector field $X$ in $\mathcal{H}(2,\sigma)$ verifies

\begin{equation}\label{eq:spectrum dissipative}
\mbox{\rm  Spc}(X)\subset\Big\{z\in\mathbb{C}:\Re(z)\leq0\Big\}\setminus\big\{(0,0)\big\},
\end{equation} where $Spc(X)\subset \mathbb{C}$ is the set of eigenvalues of $DX(z)$, for all $z\in \mathbb{R}^2\setminus{\overline{D}_\sigma}$.
\par
In this section we study  the case of dissipative vector fields in $\mathcal{H}(2,\sigma)$. This is complemented in the next section, where we focus on vector fields that are free of real eigenvalues. Observe that the divergence of every vector field  verifying \eqref{eq:spectrum dissipative} is Lebesgue almost--integrable on $\mathbb{R}^2\setminus{\overline{D}_\sigma}$, in the sense of \cite{MR2287882}.
\par
Recall that, a map $\widetilde{X}:\mathbb{R}^2\to \mathbb{R}^2$ is a local homeomorphism (respectively a local diffeomorphism) if for every $z\in\mathbb{R}^2$ there exist open neighborhoods $U\subset \mathbb{R}^2$ of $z$ y $V\subset \mathbb{R}^2$ of $\widetilde{X}(z)$ such that the mapping
$
U\ni p\mapsto\widetilde{X}(p)\in V
$ 
is a continuous (resp. differentiable) bijection whose inverse is also continuous (resp. differentiable).


\subsection{On extensions of  dissipative vector fields}\label{sec:extensions}%
Consider the family of vector fields $$\{X_{\mu}(z)= X(z)+ \mu z \; : \; \mu \in \Real \},$$ where $X\in \mathcal{H}(2,\sigma)$ and $X$ has some singularity in $\mathbb{R}^2\setminus{\overline{D}_\sigma}$. In this section, we also prove the existence of  globally injective extensions of $X_{\mu}: \Real^2\setminus{\overline{D}_{\sigma}}\to \Real^2$  to the whole plane whose divergence is Lebesgue almost--integrable in $\mathbb{R}^2$.
\par
Consider $X=(f,g) \in  \mathcal{H}(2,\sigma)$. Here, $\mathcal{F}(h)$ with $h\in\{f,g,\tilde{f},\tilde{g}\}$ denotes the continuous foliation given by the level sets $\{h=\mbox{constant}\}$ and then the leaves of the foliations are differentiable.
The concept of half-Reeb component is a natural generalization of the description of the planar foliations given in \cite{MR0089412}.
More precisely, $\mathcal{A}$  is a \textit{half--Reeb component} of $\mathcal{F}(h)$ if there is a homeomorphism
\[
H: B=\Big\{(x,y)\in [0,2]\times[0,2]:0<x+y\leq 2 \Big\}\to\mathcal{A},
\]
which is a topological equivalence between $\mathcal{F}(h) \vert_{\mathcal{A}}$ and ${\mathcal{F}}(h_0) \vert_B$ with $h_0(x,y)=xy$ such that:
\begin{itemize}
    \item The segment $\{(x,y)\in B : x+y=2\}$ is sent by $H$ onto a transversal section for the foliation $\mathcal{F}(h)$ in the complement of the point $H(1,1)$.
    \item Both segments $\{(x,y)\in B : x=0 \}$ and $\{(x,y)\in B : y=0\}$ are sent by $H$ onto full half--leaves of $\mathcal{F}(h)$.
\end{itemize}

\begin{lem}\label{lem: existence of the index} If $X\in\mathcal{H}(2,\sigma)$ has some singularity,  the following hold:
\begin{itemize}
  \item[(a)] $\exists s_0\geq\sigma$ and a globally injective local homeomorphism $\widetilde{X}:\mathbb{R}^2\to\mathbb{R}^2$ such that $\widetilde{X}(0)=0$ and $\widetilde{X}$ and $X$ coincide on $\mathbb{R}^2\setminus{\overline{D}_{s_0}}$;

  \item[(b)] the index $\mathcal{I}(X)$ at infinity is a well defined number in $[-\infty,+\infty)$.
\end{itemize}

\end{lem}

\begin{proof}
Observe that every $X=(f,g)\in\mathcal{H}(2,\sigma)$ satisfies $\mbox{\rm  Spc}(X)\cap[0,+\infty)=\emptyset$ and $\mbox{\rm  Spc}(X)\subset\{z\in\mathbb{C}:\Re(z)\leq0\}$. Thus, the methods of \cite[Proposition 2]{MR3062761} lead us to obtain that:
\begin{itemize}
\item[(a.1)] Any half--Reeb component of either $\mathcal{F}(f)$ or $\mathcal{F}(g)$ is  bounded.
\end{itemize}

Under conditions (a.1) and $\mbox{\rm  Spc}(X)\cap[0,+\infty)=\emptyset$, the means developed in the last section of \cite{MR2266382} (see Proposition 5.1) can be applied. Consequently, the proof of \cite[Theorem 3]{MR3062761} gives the existence of a closed curve $C$, surrounding the singularity of $X$ and the point at the origin, so that $C$ is  embedded in the plane, and it admits an exterior collar neighborhood $U\subset\mathbb{R}^2\setminus D(C)$ such that:

\begin{itemize}
    \item [(a.2)] $X(C)$ is a non--trivial closed curve surrounding the origin, $X(U)$ is an exterior collar
    neighborhood of $X(C)$ and the restriction $X|_{U}:U\to X(U)$ is a homeomorphism.
    \item[(a.3)] The foliation $\mathcal{F}(f)$, restricted to $\mathbb{R}^2\setminus D(C)$ is topologically equivalent to the foliation made up by all the vertical straight lines on $\mathbb{R}^2\setminus D_1$.
\end{itemize}

By Schoenflies Theorem, \cite{MR1535106,MR728227,MR2190924}, the map $X|_{C}:C\to X(C)$ can be extended to a homeomorphism ${X_1}:{\overline{D}(C)}\to{\overline{D}(X(C))}$ with ${X_1}(0)=0$.
We also extend $X|_{\mathbb{R}^2\setminus D(C)}:\mathbb{R}^2\setminus D(C)\to \mathbb{R}^2$ to $\widetilde{X}=(\tilde{f},\tilde{g}):\mathbb{R}^2 \to \mathbb{R}^2$ by defining $\widetilde{X}{|_{\overline{D}(C)}}=X_1.$
Thus, $\widetilde{X}|_{U}:U\to X(U)$ is a homeomorphism, and $U$ (resp. $X(U)$) is a exterior collar neighborhood of $C$ (resp. $X(C)$).
\par
Furthermore, (a.3) implies that $\widetilde{X}$ is a local homeomorphism whose foliation $\mathcal{F}(\tilde{f})$ is trivial.
Therefore, $\tilde{X}$ is injective by \cite[Proposition 1.4]{MR2096702}. Observe that $\tilde{X}(0)=0$ by construction. Hence $(a)$ holds considering $s_0\geq\sigma$ with $D(C)\subset\overline{D}_{s_0}$.
\par
Under these conditions, \cite[\mbox{Theorem 11}]{MR2287882} gives the existence of some $r>s_0$ such that the restriction  $X|_{\mathbb{R}^2\setminus \overline{D}_{r}}:{{\mathbb{R}^2\setminus \overline{D}_{r}}}\to \mathbb{R}^2$ admits a global differentiable extension $\widehat{X}$ which divergence is Lebesgue almost--integrable on $\mathbb{R}^2$ and $\widehat{X}(0)=0$. Consequently, the index of $X$ at infinity  $\mathcal{I}(X)$ is a well defined number of the extended real line $\in[-\infty,+\infty)$. This concludes the proof.
\end{proof}

\begin{cor}\label{cor:3.2} If $X\in\mathcal{H}(2,\sigma)$ has some singularity and $\mu\leq 0$,  then $\exists s_{\mu}>\sigma$ such that:
\begin{itemize}
  \item[(a)] the restriction  $X|_{{\mathbb{R}^2\setminus \overline{D}_{s_{\mu}}}}:{{\mathbb{R}^2\setminus \overline{D}_{s_{\mu}}}}\to \mathbb{R}^2$ satisfies the statement of \mbox{Lemma \ref{lem: existence of the index}};

\item[(b)] the vector field $X_{\mu}$ has a well--defined index at infinity and
\[\mu<0\Rightarrow\mathcal{I}(X_{\mu})<0.\]
\end{itemize}
\end{cor}

\begin{proof} Observe that a direct computation of the eigenvalues in $X_{\mu}(z)=X(z)+\mu z$ shows that
\begin{equation}\label{eq: translation spectrum}
\mbox{Spc}(X_{\mu})=\mu+\mbox{Spc}(X).
\end{equation}
Thus, as a consequence of \eqref{eq: translation spectrum}, we can obtain that:
\begin{equation}\label{eq:strong dissipative}
\mbox{Spc}(X_{\mu})\subset\{z\in\mathbb{C}:\Re(z)\leq\mu\}.
\end{equation}
Therefore,  for every $\mu\leq0$,  (a) follows  by \mbox{Lemma \ref{lem: existence of the index}}.

Item (b) holds by \eqref{eq:strong dissipative}  and definition of index at infinity of $X_{\mu}$, provided by $$\mbox{Trace}(DX_{\mu})\leq\mu<0.$$
\end{proof}

\begin{cor}\label{cor:3.2BIS} Let $X\in\mathcal{H}(2,\sigma)$ with some singularity. Suppose in addition that
\begin{equation} \nonumber 
\mbox{det}(DX_{\mu})>0,\mbox{ for all $\mu$ in some open interval } (0,\epsilon_0).
\end{equation}
Then, $\exists \varepsilon >0$  such that for every $\mu\in [-\varepsilon, \varepsilon]$ there exists $s_{\mu}>\sigma$ for which the restriction $X|_{{\mathbb{R}^2\setminus \overline{D}_{s_{\mu}}}}:{{\mathbb{R}^2\setminus \overline{D}_{s_{\mu}}}}\to \mathbb{R}^2$ satisfies the statement of Lemma \ref{lem: existence of the index}. 
\end{cor}

\begin{proof}

By hypothesis, $\mbox{det}(DX_{\mu})>0$ for all $\mu\in(0,\epsilon_0)$,  then \eqref{eq: translation spectrum}  implies that
\[
   \mbox{Spc}(X_{\mu}) \cap\Big(-\frac{\epsilon_0}{4},+\infty\Big)=\emptyset, \quad\mbox{when} \quad \mu\in\Big[-\frac{\epsilon_0}{4},\frac{\epsilon_0}{4}\Big].
\]

Consider $\varepsilon=\frac{\epsilon_0}{4}>0$. Thus, Proposition 2.7 in \cite{MR2266382} implies that, for all $\mu\in[-\varepsilon,\varepsilon]$, $X_{\mu}=(f_{\mu},g_{\mu})$ has the following property: \lq\lq Any half--Reeb component of either $\mathcal{F}(f_{\mu})$ or $\mathcal{F}(g_{\mu})$ is  bounded\rq\rq.
Therefore, the proof of Lemma \ref{lem: existence of the index} holds verbatim.
\end{proof}

\begin{rem}\label{rem:main} Under conditions of Corollary \ref{cor:3.2BIS}, the set $\big\{s_{\mu}\geq\sigma:\mu\in[-\epsilon,\epsilon]\big\}$  is a well defined set. In addition, this set is bounded when $\mu\mapsto s_{\mu}$ is continuous. 
\end{rem}

When the extension given by Lemma \ref{lem: existence of the index}, Corollary \ref{cor:3.2} and Corollary \ref{cor:3.2BIS} is also a local diffeomorphism, we say that the vector field $\widetilde{X}_{\mu}:\Real^2 \to \Real^2$ is a \textit{strong extension} of the vector field $X_{\mu}:\Real^2\setminus \overline{D}_{s_\mu}\to \Real^2$.


\subsection{Global vector fields}\label{sec: global case} In this section we study the stability at infinity of the family of global  vector fields $$\displaystyle{\{Y_{\mu}(z)= Y(z)+ \mu z  \; : \;\mu\in \Real\}},$$ where $Y: \Real^2\to \Real^2$ is a  vector field  such that $Y(0)=0$ and the restriction $Y|_{\mathbb{R}^2\setminus{\overline{D}_{\sigma}}}:\mathbb{R}^2\setminus{\overline{D}_{\sigma}}\to\mathbb{R}^2$ belongs to $\mathcal{H}(2,\sigma)$.

\begin{prop}\label{prop:without sing}%
 If there exists $\varepsilon_0>0$ such that the map $Y_{\mu}(z)=Y(z)+\mu z$ is a local diffeomorphism, for all $\mu\in(-\varepsilon_0,\varepsilon_0)$. Then, $\exists \varepsilon>0$ such that the restriction $Y_{\mu}|_{\mathbb{R}^2\setminus{\overline{D}_{\sigma}}}:\mathbb{R}^2\setminus{\overline{D}_{\sigma}}\to\mathbb{R}^2$ has no singularities in $\mathbb{R}^2\setminus{\overline{D}_{\sigma}}$, for every $\mu\in[-\varepsilon,\varepsilon]$.
\end{prop}

\begin{proof}

As a consequence of \eqref{eq: translation spectrum}, the following assumptions are equivalent

\begin{itemize}
\item[(i)] $\displaystyle{0\not\in\mbox{Spc}(Y_{\mu}) \quad \text{for all} \; \mu\in[0,\varepsilon_0) \quad (\text{resp. \; for all} \; \mu\in(-\varepsilon_0,0]),}$

\item[(ii)] $\displaystyle{ \mbox{Spc}(Y)\cap(-\varepsilon_0,0]=\emptyset \quad (\text{resp.} \quad \mbox{Spc}(Y)\cap[0,\varepsilon_0)=\emptyset).}$

\end{itemize}

Since the restriction $Y|_{\mathbb{R}^2\setminus{\overline{D}_{\sigma}}}:\mathbb{R}^2\setminus{\overline{D}_{\sigma}}\to\mathbb{R}^2$ belongs to $\mathcal{H}(2,\sigma)$,  condition \eqref{eq: translation spectrum} implies  that

$$\mbox{Spc}(Y_\mu|_{\Real^2\setminus \overline{D}_{\sigma}})\bigcap\left(-\frac{\varepsilon_0}{4},\frac{\varepsilon_0}{4}\right)=\emptyset, \; \forall \mu \in [-\frac{\varepsilon_0}{2},\frac{\varepsilon_0}{2}].$$
Hence,  by \cite[Theorem 2.1]{MR2266382}, the following holds:
\begin{equation}\label{eq:Thm C gut}
Y_{\mu}:\mathbb{R}^2\to\mathbb{R}^2 \quad\mbox{is globally injective, for every }\mu\in\Big[\displaystyle-\frac{\varepsilon_0}{2},\frac{\varepsilon_0}{2}\Big],
\end{equation}
(see also \cite[Theorem 3]{MR2609212} and \cite{MR1045818}).
\par
Consequently, since $Y_{\mu}(0)=0$, \eqref{eq:Thm C gut} implies that there exists $\varepsilon=\frac{\varepsilon_0}{2}>0$ such that the restriction $Y_{\mu}|_{\mathbb{R}^2\setminus{\overline{D}_{\sigma}}}:\mathbb{R}^2\setminus{\overline{D}_{\sigma}}\to\mathbb{R}^2$ has no singularities in $\mathbb{R}^2\setminus{\overline{D}_{\sigma}}$, for every $\mu\in[-\varepsilon,\varepsilon]$. Thus, the proposition holds.
\end{proof}

\begin{prop}\label{prop:main global}%
Let  $\varepsilon >0$ be as in Proposition \ref{prop:without sing}. Suppose in addition that $\exists s>\sigma$ such that, for each $0<\mu\leq\varepsilon$,

 \begin{itemize}
  \item[(1)] $Y_{\mu}$ induces a well--defined negative semi--flow on $\mathbb{R}^2\setminus\overline{D}_s$,
  \item[(2)] the index $\mathcal{I}(Y_\mu)\in[-\infty,+\infty)$ satisfies $\mu\cdot\mathcal{I}(Y_\mu)>0$.
\end{itemize}
Then, for each $\mu\in[-\varepsilon,\varepsilon]$, $\infty$  is an attractor (respectively a repellor) for the vector field $Y_{\mu}|_{\mathbb{R}^2\setminus{\overline{D}_{s}}}:\mathbb{R}^2\setminus{\overline{D}_{s}}\to\mathbb{R}^2$ provided by  $\mathcal{I}(Y_\mu)>0$ (respectively $\mathcal{I}(Y_\mu)<0$).
\end{prop}

\begin{proof}
Observe that, for all $-\varepsilon\leq \mu<0$

$$\mbox{Spc}(Y_{\mu}|_{\mathbb{R}^2\setminus\overline{D}_{\sigma}})\subset \{z\in \mathbb{C} : \Re(z)\leq \mu<0\}\setminus (-\frac{\varepsilon}{2},0]$$
Hence, for all $s_0 \geq \sigma$ the following hold:

\begin{itemize}
\item[(a.1)] $\mbox{\rm Trace}\big(DY_{\mu}\big)<0$ and so $\;z\mapsto\mbox{\rm Trace}\big(DY_{\mu}(z)\big)$ is Lebesgue almost--integrable in $\mathbb{R}^2\setminus\overline{D}_{s_0}$(\cite[\mbox{Lemma 7}]{MR2287882});
\item[(a.2)] $\mathcal{I}(Y_{\mu})$ is a well-defined number in $[-\infty,+\infty)$ (\cite[\mbox{Corollary 13}]{MR2287882}).
\item[(a.3)] for every $z\in \mathbb{R}^2\setminus\overline{D}_{s_0}$, there is only one
positive semi--trajectory of $Y_{\mu}$ passing through $z$.
(\cite[\mbox{Teorema 4}]{MR2287882}).
\end{itemize}
Therefore,  $\mathcal{I}(Y_{\mu})$ is a well--defined number and $\mathcal{I}(Y_{\mu})\neq 0$,  for all $\mu \in [-\varepsilon, \varepsilon]$ .
\begin{itemize}
  \item[(a.4)] We claim that for every $\mu\in[-\varepsilon,\varepsilon]$, the condition $\mathcal{I}(Y_\mu)\neq0$ implies that $Y_{\mu}$ has a bounded set of periodic trajectories.
\end{itemize}
Since $Y_{\mu}(0)=0$ and $Y_{\mu}$ is injective by \eqref{eq:Thm C gut}, we  proceed by contradiction:
\begin{itemize}
  \item [(*)] We suppose that $\{\Gamma_1,\Gamma_2,\dots,\Gamma_n,\dots\}$ is an unbounded set of periodic trajectories of $Y_{\mu}$ such that
    \[
    \overline{D}(\Gamma_1)\subset\overline{D}(\Gamma_2)\subset\cdots\subset\overline{D}(\Gamma_n)\subset\cdots
    \]
\end{itemize}
Under these conditions, by using the Green Theorem  in $\overline{D}(\Gamma_n)$ and the arc length element $ds,$ we obtain that
\[
{\int_{\overline{D}(\Gamma_n)}}\mbox{\rm
Trace}(DY_{\mu})dx\wedge dy=\oint_{\Gamma_n}\langle Y_{\mu}(s),\eta_n^{e}(s)\rangle ds,
\]
where $\eta_n^{e}(p)$ is the unitary outer normal vector to $\Gamma_n$ and $\langle Y_{\mu}(s),\eta_n^{e}(s)\rangle$ is the inner product of $Y_{\mu}(s)$ with $\eta_n^{e}(s)$.
Thus,
\[
\mathcal{I}(Y_{\mu})=\lim_{n\to\infty}{\int_{\overline{D}(\Gamma_n)}}\mbox{\rm
Trace}(DY_{\mu})dx\wedge dy=0,
\]
provided by $\langle Y_{\mu}(s),\eta_n^{e}(s)\rangle=0$, for all $\eta_n^{e}(s)$. This contradiction with $\mathcal{I}(Y_\mu)\neq0$ shows that (*) never happens, and gives the proof of (a.4).

Furthermore, since $[-\varepsilon,\varepsilon]$ is a compact set, the following is directly obtained applying (a.4):

\begin{itemize}
  \item[(a.5)] $\exists s>s_0>\sigma$ such that $Y_{\mu}$ has no periodic trajectories in $\mathbb{R}^2\setminus\overline{D}_s$ as long as $\mu\in(-\varepsilon,\varepsilon)$ and $\mathcal{I}(Y_\mu)\neq0$.
\end{itemize}

Hence, $Y_\mu$ generates a positive semi--flow on $\mathbb{R}^2\setminus\overline{D}_s$, for all $-\varepsilon<\mu<0$. In addition, for $s>\sigma$ as in (a.5) we have that for all $\mu\in(-\varepsilon,\varepsilon)$ the following hold:

\begin{itemize}
\item[(a.6)] $Y_\mu(0)=0$;
\item[(a.7)] $\mbox{det}(DY_{\mu}(z))>0$, $\forall z\in\mathbb{R}^2$, so $Y_\mu$ preserves orientation;
\item[(a.8)] there exists $c>0$ such that $||Y_\mu(z)||>c$, for any $z\in \mathbb{R}^2\setminus{\overline{D}_{s}}$, by the openness of $Y_\mu$.

\end{itemize}

Under these properties: \eqref{eq:Thm C gut} and  (a.1) to (a.8), the vector field  $Y_{\mu}$ satisfies all the conditions of \cite[Theorem 26]{MR2287882}. Consequently:
\begin{itemize}
\item[(a.9)]
    For every $r\geq s$ there exist a closed curve $C_r$ transversal to $Y_{\mu}$ contained in the regular set $\mathbb{R}^2\setminus \overline{D}_r$. In particular, $D(C_r)$ contains $D_r$ and $C_r$ has transversal contact to each small local integral curve of $Y$ at any $p\in C_r$.
\end{itemize}
Moreover, for  $r\geq s$ large enough the closed curve $C_r\subset\mathbb{R}^2$  is transversal to $Y_{\mu}$, near $\infty$ and it is such that $\mu\in(-\varepsilon,\varepsilon)$, the index $\mathcal{I}(Y_{\mu})$ and
\[
{\int_{\overline{D}(C_r)}}\mbox{\rm Trace}(DY_{\mu})dx\wedge dy=\oint_{C_r}\langle Y_{\mu}(s),\eta_n^{e}(s)\rangle ds
\]
have the same sign. Thus, the point at infinity of the Riemann sphere $\mathbb{R}^2\cup\{\infty\}$ is either an attractor or a repellor of $Y_{\mu}:\mathbb{R}^2\setminus\overline{D}_s\to\mathbb{R}^2$.
Actually, by (a.9) and \cite[Theorem 28]{MR2287882} we obtain that:
\begin{enumerate}
\item[(a.10)] If $\mathcal{I}(Y_{\mu})<0$ (respectively $\mathcal{I}(Y_{\mu})>0$), then $\infty$ is a repellor (respectively  an attractor) of the vector field $Y_{\mu}:\mathbb{R}^2\setminus\overline{D}_s\to\mathbb{R}^2$.
\end{enumerate}
It concludes the proof, and the proposition holds.
\end{proof}


\subsection{Hopf bifurcation at infinity for dissipative vector fields} \label{subsec: hopf bif}%

We are now ready to state our main theorem.

\begin{thm}\label{main 1}%
Consider the family $\displaystyle{\{X_{\mu}(z)=X(z)+\mu z \; : \; \mu\in\mathbb{R}\}}$ where  $X$ is a differentiable vector field in $\mathcal{H}(2,\sigma)$ with some singularity and verifying
\begin{equation} \label{eq:det} 
\mbox{det}(DX_{\mu})>0,\mbox{ for all $\mu$ in some open interval } (0,\epsilon_0).
\end{equation}

Suppose in addition that

\begin{enumerate}

  \item $\exists \varepsilon_0>0$ such that the set $\big\{s_{\mu}\geq\sigma:\mu\in[-\varepsilon_0,\varepsilon_0]\big\}$ induced by strong extensions is bounded;
  \item for each $\mu>0$, $X_{\mu}$ induces a well defined negative semi--flow  and the index $\mathcal{I}(X_{\mu})\in[-\infty,+\infty]$ satisfies $\mu \cdot\mathcal{I}(X_{\mu}) >0$.
\end{enumerate}
Then there are $\varepsilon>0$ and $s> \sigma$ such that the family of vector fields
\[\Big\{X_{\mu}:\mathbb{R}^2\setminus{\overline{D}_s}\to\mathbb{R}^2\; ;\; -\varepsilon<\mu<\varepsilon\Big\}\]
has at $\mu=0$ a Hopf bifurcation at $\infty$.
\end{thm}

\begin{proof}
By the assumption (1) there exist a number $\tilde{s}=\sup\{s_{\mu}\geq\sigma:-\epsilon_0\leq\mu\leq\epsilon_0\}$ and  a local diffeomorphism $\widehat{X}_0:\mathbb{R}^2\to\mathbb{R}^2$ with $\widehat{X}_0(0)=0$ such that
\begin{equation}\nonumber 
\widehat{X}_0(z)+\mu z=X_{\mu}(z),\quad\mbox{for all}  \; z\in\mathbb{R}^2\setminus\overline{D}_{\tilde{s}} \quad\mbox{and } \; \mu\in(-\epsilon_0,\epsilon_0).
\end{equation}

In addition, the map $\widehat{X}_\mu:\Real^2\to \Real^2$ such that $\widehat{X}_\mu(z)=\widehat{X}_0(z)+\mu z$ is also a local diffeomorphism. Therefore, there exist $s\geq\tilde{s}\geq\sigma$ and $0<\varepsilon<\epsilon_0$ such that for every $\mu\in[-\varepsilon,\varepsilon]$:

\begin{itemize}
\item The restriction $X_{\mu}|_{\mathbb{R}^2\setminus{\overline{D}_{s}}}:\mathbb{R}^2\setminus{\overline{D}_{s}}\to\mathbb{R}^2$ has no singularities in $\mathbb{R}^2\setminus{\overline{D}_{s}}$. By Proposition \ref{prop:without sing}.
\item $\mu \cdot I(X)>0$, for all $\mu<0$. By Corollary \ref{cor:3.2}.
\item For $\mu<0$, the restriction $X_{\mu}|_{\mathbb{R}^2\setminus{\overline{D}_{s}}}$ has a repellor at $\infty$, and for $\mu>0$,  the restriction $X_{\mu}|_{\mathbb{R}^2\setminus{\overline{D}_{s}}}$ has an attractor at $\infty$.  By hypotheses (2) and Proposition \ref{prop:main global}.
\end{itemize}
Thus,
$\big\{X_{\mu}:\mathbb{R}^2\setminus{\overline{D}_s}\to\mathbb{R}^2;-\varepsilon<\mu<\varepsilon\big\}$
has at $\mu=0$ a Hopf bifurcation at $\infty$.
\end{proof}

Observe that the vector fields $X_{\mu}(z)=X(z)+\mu z$ as in \mbox{Theorem \ref{main 1}} satisfy
\[
X_{\mu}(z)-X_{\tilde{\mu}}(z)=(\mu-\tilde{\mu})z, \quad \forall z\in \mathbb{R}^2\setminus{\overline{D}_\sigma}.
\]
Therefore, if $\mathfrak{X}(\mathbb{R}^2\setminus{\overline{D}_\sigma})$ is the space of the continuous vector fields of $\mathbb{R}^2\setminus{\overline{D}_\sigma}$ endowed with the topology induced by the uniform convergence on compact sets, the functional
\[
\mathbb{R}\ni\mu \mapsto X_{\mu}\in\mathfrak{X}(\mathbb{R}^2\setminus{\overline{D}_\sigma})
\]
is continuous.

\begin{rem}\label{rem:discontinuity}
If the maps $h_{\mu}:\mathbb{R}^2\setminus{\overline{D}_\sigma}\to(0,+\infty)$ and $\tilde{h}_{\mu}:\mathbb{R}^2\setminus{\overline{D}_\sigma}\to(-\infty,0)$ are differentiable and
the family  $\big\{X_{\mu}:\mathbb{R}^2\setminus{\overline{D}_s}\to\mathbb{R}^2;-\varepsilon<\mu<\varepsilon\big\}$ is given by \mbox{Theorem \ref{main 1}}, then the families

\[
h_{\mu}(z)X_{\mu}(z)
\quad
\mbox{and}
\quad
\tilde{h}_{\mu}(z)X_{\mu}(z)
\]
also have at $\mu=0$ a Hopf bifurcation at $\infty$. For instance, the map
\[
h_{\mu}(z)=\left\{
             \begin{array}{ll}
               \displaystyle\frac{1}{\mu}, & \hbox{ if } \mu\neq0; \\
               1, & \hbox{ if } \mu=0
             \end{array}
           \right.
\]
Notice that $h_{\mu}$ induces a well--defined map  $\mu \mapsto h_{\mu} X_{\mu}$ which is  discontinuous at $\mu=0$.
\end{rem}
In some sense, \mbox{Remark \ref{rem:discontinuity}} shows that \mbox{Theorem \ref{main 1}} includes the results where the strong domination imposed by the linear part is used.
\par
Furthermore, we do not assume the existence of some open neighborhood of infinity free of singularities, as in  Theorem \ref{teoAGG07} item (4). Recall that the results in this work are obtained from the weak hypothesis on the Jacobian determinant like \eqref{eq:det}, which is natural in order to work with isolated singularities. In addition, the assumption (2) in Theorem \ref{main 1} is necessary,  since the theorem considers the differentiable vector fields, not necessarily of class $C^1$.

In the particular case of $C^1-$vector fields, Theorem \ref{main 1} directly gives the next illustrative result.

\begin{cor}\label{cor:1}%
Consider  $Z\in\mathcal{H}(2,\sigma)$ of class $C^1$. Suppose $Z$ has a singularity  and the maps $\mathbb{R}^2\setminus{\overline{D}_\sigma}\ni z\mapsto  Z_{\mu}(z)=Z(z)+\mu z$ are orientation preserving local diffeomorphisms, for all $\mu$ in some interval $(0,\epsilon_0)$.

If assumption (1) in Theorem \ref{main 1} holds, then the condition
\begin{equation}\label{eq:index}
  \mu\neq0\Rightarrow\mu \cdot\mathcal{I}(Z_{\mu}) >0
\end{equation}
implies the existence of $\varepsilon>0$ and $s> \sigma$ such that
$\Big\{Z_{\mu}:\mathbb{R}^2\setminus{\overline{D}_s}\to\mathbb{R}^2;-\varepsilon<\mu<\varepsilon\Big\}$
has at $\mu=0$ a Hopf bifurcation at $\infty$.
\end{cor}

Observe that condition \eqref{eq:index} extends the case where the spectrum of $X_{\mu}$ crosses the imaginary axis transversally when $\mu$ moves from negative to positive values.


\section{Vector fields free of real eigenvalues} \label{sec:free eigen}%

This subsection addresses the special case of differentiable vector fields whose Jacobian matrix is free of real eigenvalues. These vector fields induce local diffeomorphisms with zero divergence.

\begin{prop} \label{cor:3.3}
Let $X:\mathbb{R}^2\setminus{\overline{D}_\sigma}\to\mathbb{R}^2$ be a differentiable vector field with some singularity. Suppose that the following holds
\begin{equation} \label{eq:spectrum free divergence}
\mbox{\rm  Spc}(X)\subset\Big\{z\in\mathbb{C}:\Re(z)=0\Big\}\setminus\big\{(0,0)\big\}.
\end{equation}
Then, $\forall \mu\in \Real$  $\exists s_{\mu}\geq \sigma$  such that
\begin{itemize}
\item[(a)] the restriction  $X|_{{\mathbb{R}^2\setminus \overline{D}_{s_\mu}}}:{{\mathbb{R}^2\setminus \overline{D}_{s_\mu}}}\to \mathbb{R}^2$ is injective and admits a global differentiable extension $\widehat{X}:\Real^2\to \Real^2$ with $\widehat{X}(0)=0$ and $$\widehat{X}(z)=X(z), \;\forall z\in \mathbb{R}^2\setminus \overline{D}_{s_\mu};$$
  \item[(b)] the index $\mathcal{I}(X)$ at infinity is a well defined number in $[-\infty,+\infty)$;
\item [(c)] $X_{\mu}$ induces a well defined positive (respectively negative) semi--flow as long as $\mu<0$ (respectively $\mu>0$). Moreover, the trajectories of $X$ are unique in the sense that only depend of the initial condition.
\end{itemize}
\end{prop}

\begin{proof}
Since $X_{\mu}(z)=X(z)+\mu z$ implies that
$\mbox{\rm  Spc}(X_\mu)\cap\mathbb{R}=\emptyset$ for all $\mu\in\mathbb{R}$. Proposition 2.7 in \cite{MR2266382}, implies that $X_{\mu}=(f_{\mu},g_{\mu})$ has the following property:
\begin{itemize}
\item[(a.1)] Any half-Reeb component of either $\mathcal{F}(f_{\mu})$ or $\mathcal{F}(g_{\mu})$ is a bounded set.
\end{itemize}
Therefore, the proof of Lemma \ref{lem: existence of the index} holds verbatim and statements (a) and (b) remain true.

In addition, \eqref{eq: translation spectrum} implies that:
\[
\begin{array}{ccl}
  \mu<0 & \Rightarrow & \mbox{Spc}(X_{\mu}) \subset \big\{z\in\mathbb{C}:\Re(z)\leq\mu<0\big\},\\
  \mu>0 & \Rightarrow & \mbox{Spc}(-X_{\mu}) \subset \big\{z\in\mathbb{C}:\Re(z)\leq-\mu<0\big\}.
\end{array}
\]
Therefore, \cite[Lemma 3.3]{MR2096702} implies that the vector fields $X_{\mu}$ with $\mu<0$, and $-X_{\mu}$ with $\mu>0$, induce a well defined positive semi--flow.
Similarly, \cite[Lemma 2.1]{MR2448587} gives the uniqueness of the trajectories induced by $X$.
Thus, statement (c) holds.
\end{proof}

\begin{prop}\label{prop:div 0}
Consider the family of differentiable vector fields $\{X_{\mu}(z)=X(z)+\mu z : \mu\in \Real\}$, where $X$ verifies hypotheses of Proposition \ref{cor:3.3}.

If $\exists \varepsilon_0>0$ such that the set $\big\{s_{\mu}\geq\sigma:\mu\in[-\varepsilon_0,\varepsilon_0]\big\}$ induced by strong extensions is bounded, then $\exists s>\sigma$ and $\varepsilon>0$ such that the family
\[\big\{X_{\mu}:\mathbb{R}^2\setminus{\overline{D}_s}\to\mathbb{R}^2;-\varepsilon<\mu<\varepsilon\big\}\]
    has at $\mu=0$ a Hopf bifurcation at $\infty$.
\end{prop}

\begin{proof}Follows by Proposition \ref{cor:3.3} and Proposition \ref{prop:main global}.
\end{proof}

Observe that a similar result as described in Remark \ref{rem:discontinuity} is obtained in the case of families given by \mbox{Proposition \ref{prop:div 0}}.

\begin{figure}[htb]
  \centering
  \psfrag{D}{$\overline{D}_s$}
   \psfrag{In}{$\infty$}
    \psfrag{c}{$\mu=0$}
     \psfrag{m}{$\mu<0$}
      \psfrag{M}{$\mu>0$}
  \includegraphics[width=12cm]{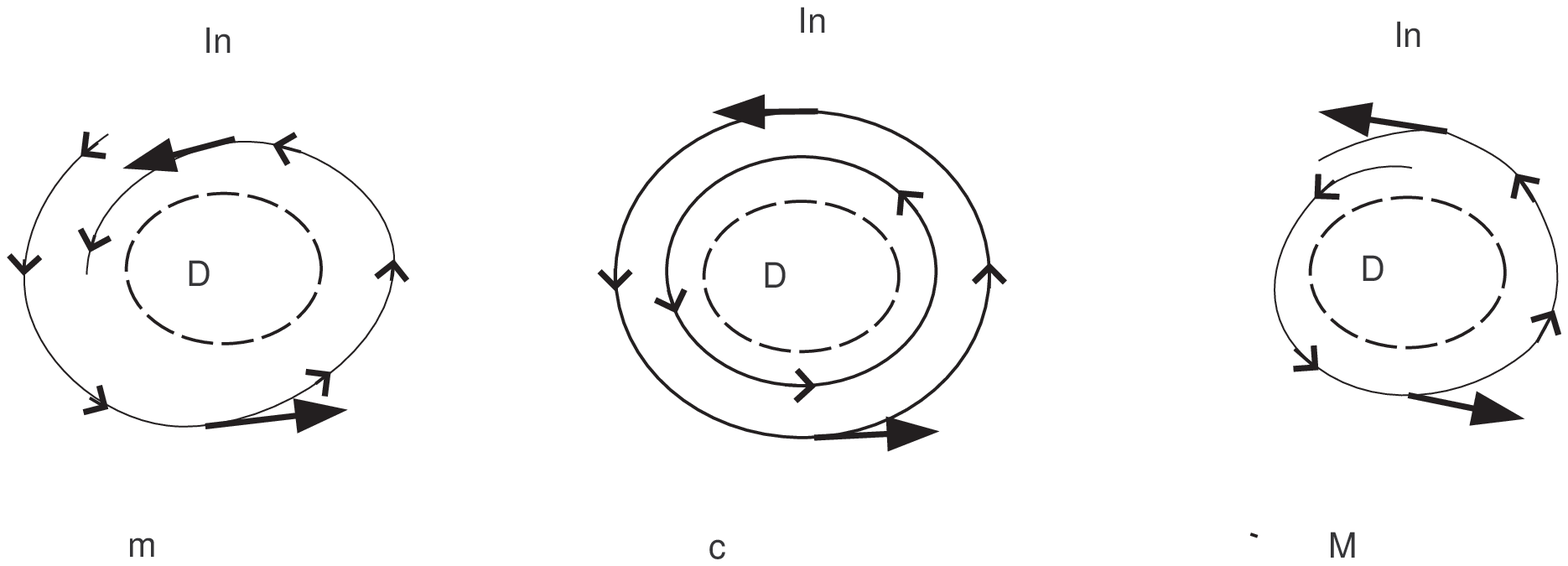}\\
  \caption{\footnotesize The bifurcation in \mbox{Proposition \ref{prop:div 0}}.}\label{fig:mainf proposition}
\end{figure}

\begin{rem}\label{rem:center}
Consider now the special case of planar $C^1-$vector fields $Y:\mathbb{R}^2\to\mathbb{R}^2$ such that  $Y(0)=0$ and
\[
\mbox{\rm  Spc}(Y)\subset\Big\{z\in\mathbb{C}:\Re(z)=0\Big\}\setminus\big\{(0,0)\big\}.
\]
The trajectories in $\mathbb{R}^2\setminus\{0\}$ induced by $Y$ are periodic orbits surrounding the origin \cite{MR2448587,MR3223368}.
Moreover, a direct application of \mbox{Proposition \ref{prop:div 0}} gives the existence of some $\varepsilon>0$, for which the family $\big\{Y_{\mu}:\mathbb{R}^2\setminus{\overline{D}_s}\to\mathbb{R}^2;-\varepsilon<\mu<\varepsilon\big\}$ has at $\mu=0$ a Hopf bifurcation at $\infty$. See Figure \ref{fig:mainf proposition}.
\end{rem}


\section*{Acknowledgements}
The first author was partially supported by grants \textsc{micinn-12-mtm2011-22956 } from Spain  and  CNPq 474406/2013-0 from Brazil.
The second author was partially supported by  Pontif\'{\i}cia Universidad Cat\'{o}lica del Per\'{u} (\textsc{dgi}:70242.0056), by Instituto de Ci\^encias Matem\'aticas e de Computa\c{c}\~ao (\textsc{icmc--usp}: 2013/16226-8).
This paper was written while the second author served as an Associate Fellow at the Abdus Salam  \textsc{ictp} in Italy. He also acknowledges the hospitality of \textsc{icmc--usp} in Brazil during the preparation of part of this work.



\def\cprime{$'$} \def\cprime{$'$}
\providecommand{\bysame}{\leavevmode\hbox to3em{\hrulefill}\thinspace}
\providecommand{\MR}{\relax\ifhmode\unskip\space\fi MR }

\providecommand{\MRhref}[2]{%
  \href{http://www.ams.org/mathscinet-getitem?mr=#1}{#2}
}
\providecommand{\href}[2]{#2}

\end{document}